\documentclass[10pt]{amsart}%
\usepackage{inputenc}
\usepackage{amsmath,amssymb}
\usepackage{fullpage}
\usepackage{amsthm}
\usepackage{multirow}
\usepackage{color}
\usepackage{enumerate}
\usepackage{tikz}
\usepackage{ytableau}
\usepackage{pifont}
\usepackage{amsmath}
\usepackage{graphicx}
  \usepackage[all]{xy}
\usepackage{amsfonts}
\usepackage{amssymb}
\usepackage{ytableau}%

\providecommand{\U}[1]{\protect\rule{.1in}{.1in}}
\providecommand{\U}[1]{\protect\rule{.1in}{.1in}}
\providecommand{\U}[1]{\protect\rule{.1in}{.1in}}
\providecommand{\U}[1]{\protect\rule{.1in}{.1in}}
\providecommand{\U}[1]{\protect\rule{.1in}{.1in}}
\newcommand{\ulambda}{{\boldsymbol{\lambda}}}

\newcommand{\umu}{{\boldsymbol{\mu}}}

\newcommand{\uemptyset }{{\boldsymbol{\emptyset}}}
\newtheorem{Th}{Theorem}[section]
\numberwithin{equation}{section}

\newtheorem{Prop}[Th]{Proposition}
\newtheorem{conj}[Th]{Conjecture}

\theoremstyle{remark}
\newtheorem{Rem}[Th]{Remark}{\rmfamily}
\theoremstyle{definition}
\newtheorem{Def}[Th]{Definition}{\rmfamily}
\newtheorem{exa}[Th]{Example}{\rmfamily}

\newcommand\blfootnote[1]{%
  \begingroup
  \renewcommand\thefootnote{}\footnote{#1}%
  \addtocounter{footnote}{-1}%
  \endgroup
}
\newtheorem{abs}[Th]{\bfseries}
\begin{document}

\title{Crystal isomorphisms and Mullineux involution II}

\author{Nicolas Jacon}
\address{Universit\'{e} de Reims Champagne-Ardennes, UFR Sciences exactes et
naturelles. Laboratoire de Math\'{e}matiques UMR CNRS 9008. Moulin de la Housse BP
1039. 51100 Reims. France.}
\email{nicolas.jacon@univ-reims.fr} 
\author{C\'edric Lecouvey}
\address{Institut Denis Poisson CNRS UMR 7013
Universit\'{e} de Tours
Parc de Grandmont
37200 Tours, France.}
\email{cedric.lecouvey@lmpt.univ-tours.fr} 

\maketitle
\date{}
\blfootnote{\textup{2020} \textit{Mathematics Subject Classification}: \textup{20C08,05E10,20C20}} 
\begin{abstract}
We present a new combinatorial and conjectural algorithm for computing the Mullineux involution for the symmetric group 
 and its  Hecke algebra. This algorithm is built on a conjectural property of crystal isomorphisms which can be rephrased in a purely combinatorial way. 
\end{abstract}


\section{Introduction}

The Mullineux involution is an important map  which has been originally defined by Mullineux \cite{Mu} in the context of the modular representation theory of the symmetric group. More generally, it can be defined  for the class of Hecke algebras of the symmetric group \cite{Br}. Let $n\in \mathbb{Z}_{>0}$ and  $e\in \mathbb{Z}_{>1}$. Let $\eta$ be a primitive $e$ root of $1$.  The Hecke algebra of the symmetric group $\mathcal{H}_n (\eta)$ is defined as the associative unital $\mathbb{C}$-algebra with generators $T_1$, \ldots, $T_{n-1}$ and the following relations:
$$\begin{array}{rcll}
(T_i-\eta)(T_i+1)&=&0& \text{for }i=1,\ldots,n-1,\\
T_i T_{i+1} T_i &=& T_{i+1} T_i T_{i+1} &  \text{for }i=1,\ldots,n-2,\\
T_i T_j&=&T_j T_i & \text{if }|i-j|>1.\\
\end{array}$$
It is known that the simple modules of this algebra are naturally labelled by the set of $e$-regular partitions $\text{Reg}_e (n)$  with rank $n$ (see \S \ref{par}  for the definition):
 $$\operatorname{Irr} (  \mathcal{H}_n (\eta))=\{ D^{\lambda} \ |\ \lambda \in \operatorname{Reg}_e (n)\}.$$
 There is a $\mathbb{C}$-algebra automorphism $\sharp$ which can be defined on the generators of  $\mathcal{H}_n (\eta)$ as follows. For all $i=1,\ldots,n-1$, we have $T_i^{\sharp}=-\eta T_i$. This automorphism induces an involution:
 $$m_e : \operatorname{Reg}_e (n) \to \operatorname{Reg}_e (n),$$
 defined as follows. For all $\lambda \in  \operatorname{Reg}_e (n)$ there exists a unique $\mu \in  \operatorname{Reg}_e (n)$ such that the module $D^{\lambda}$ twisted by $\sharp$ is isomorphic to $D^{\mu}$. Then we define $m_e (\lambda):=\mu$.  If $e$ is prime, this involution   describes the structure of  a simple $\mathbb{F}_e\mathfrak{S}_n$ -module   twisted by the sign representation.  If $e$ is sufficiently large, or more generally if $\lambda$ is an $e$-core, it is easy to see that $m_e (\lambda)$ is just the conjugate partition $\lambda'$. 
 
 The study of the Mullineux involution has a long story. A first conjectural and combinatorial description of $m_e$ (if $e$ is prime) was first given by Mullineux \cite{Mu}  and proved later by Ford and Kleshchev \cite{FK}.  Before this proof, Kleshchev gave a solution to the computation of the involution \cite{K} (see also  \cite{BO}  and \cite{BK}). This solution may be rephrased in terms of the crystal graph theory.  Other algorithms  were given by Xu \cite{Xu1,Xu2}, or more recently by  Fayers \cite{F}, and by the author \cite{JMu}. We also note that there exist different generalizations in the context of Ariki-Koike algebras \cite{F2,JL0}, affine Hecke algebras \cite{MW,JL3},  general linear groups \cite{DJ} or  rational Cherednik algebras \cite{Lo,GJN}  and they are all connected with the above one.  We also mention a recent conjecture by Bezrukavnikov on this involution in relation with  nabla operators and Haiman's $n!$ conjecture studied in \cite{DiY}. 
 
All the above algorithms for computing the Mullineux involution have a common feature: they are recursive algorithms in $n$. The algorithms to compute the Mullineux image of a partition $\lambda$ of rank $n$ requires the computation of the Mullineux involution $m_e (\mu)$ for $|\mu|<n$. 
 The aim of this paper is to present a conjectural algorithm which is  recursive in $e$. This conjecture is in fact  built on the description of the Mullineux involution by Kleshchev in terms of crystal graphs together with the concept of crystal isomorphisms described in \cite{JL}. 
 The conjecture follows in fact from a  purely combinatorial conjecture which can be described without any mention to crystals and in a very simple way. Assuming the conjecture true, it becomes possible to compute 
  $m_e$ from the datum of $m_{2e}$. As $m_e$ corresponds to the conjugation of partitions if $e$ is sufficiently large, the algorithm follows. 
 
 The paper is organized as follows.  We first recall several elementary combinatorial notions on partitions and crystals. This section ends with  a presentation of the  Kleshchev' solution to the Mullineux problem. The second section explains the notion of crystal isomorphism. We then give a conjectural combinatorial property, Conjecture \ref{conj}, which can be rephrased in the context of crystal isomorphisms. The last section presents several new results around this notion and states the conjectural algorithm for computing the Mullineux involution. \\
\\
{\bf  Acknowledgements}: The authors are grateful to Matt Fayers for useful discussions. The first author is supported by ANR project  AHA ANR-18-CE40-0001. Both authors are supported by  ANR project CORTIPOM  ANR-21-CE40-0019.

\section{Mullineux involution for Hecke algebras}
We first start with the definition of several elementary notions. Then we present the Kleshchev solution to the computation of the Mullineux involution. 

\subsection{Partitions and Young diagrams}\label{par} 
A {\it partition} is a non increasing
sequence $\lambda=(\lambda_{1},\cdots,\lambda_{m})$ of nonnegative
integers.  The {\it rank}  of the partition is by  definition the number $|\lambda|=\sum_{1\leq i\leq m} \lambda_i$. 
 We say that $\lambda$ is a partition of $n$, where $n=|\lambda|$. The unique partition of $0$ is the {\it empty partition} $\emptyset$. We denote by $\Pi (n)$ the set of partitions of $n$.  For $e\in \mathbb{Z}_{>1}$, we say that 
 $\lambda$ is an {\it $e$-regular partition} if  no non zero part of $\lambda$ can be repeated $e$ or more times. The set of $e$-regular partitions of rank $n$  is denoted by $\text{Reg}_e (n)$. 
Given a partition $\lambda\in \Pi (n)$, its \emph{Young diagram} $[\lambda]$ is the set: 
$$[\lambda] = \big\{(a,b)\ |\ 1\leq a\leq r,\ 1 \leq b\leq \lambda_a\big\} \subset \mathbb{N}\times\mathbb{N}.$$
The elements of this set are called the {\it nodes} of $\lambda$. The {\it $e$-residue} (ore more simply, residue) of a node $\gamma\in [\lambda]$ is by definition  $\operatorname{res}  (\gamma)=b-a+e\mathbb{Z}$. For $j\in \mathbb{Z}/e\mathbb{Z}$, we say that $\gamma$ is a {\it $j$-node} if $\operatorname{res} (\gamma)=j$. In addition,  $\gamma$ is called a {\it removable $j$-node} for $\lambda$ if the set $[\lambda]\setminus\{\gamma\} $ is the Young diagram of some partition $\mu$. In this case, we also say that $\gamma$ is an {\it addable $j$-node} for $\mu$. \smallskip

Let $\gamma=(a,b)$ and $\gamma'=(a',b')$ be two addable or removable $j$-nodes of the same partition $\lambda$. Then we write  $\gamma>\gamma'$ if $a<a'$. Let  $w_j (\lambda)$ be the word obtained by reading all the addable and removable $j$-nodes in increasing order and by encoding each addable $j$-node with the letter $A$ and each removable $j$-node with the letter $R$.  Then deleting as many subwords  $RA$ in this word as possible, we obtain a new word  $\widetilde{w}_j (\lambda)=A\cdots A R \cdots R$. The node corresponding to the rightmost $A$ (if it exists) is called the {\it good addable  $j$-node} and the node corresponding to the leftmost $R$ (if it exists) is called the {\it good removable $j$-node}.

\subsection{Level $1$ Fock space}\label{l1}
Let $\mathcal{F}$ be the $\mathbb{C}$-vector space with basis given by all the
   partitions. It is called the (level $1$) {\it Fock space}. There is an action of  $\mathcal{U} (\widehat{\mathfrak{sl}}_e)$ on $\mathcal{F}$  which makes $\mathcal{F}$ into an integrable module of level $1$.
For $i\in \mathbb{Z}$, the Kashiwara operators $\widetilde{e}_{i+e\mathbb{Z},e}$ and $\widetilde{f}_{i+e\mathbb{Z},e}$ are then defined as follows.
\begin{itemize}
\item If $\lambda$ has no addable $i$-node then $\widetilde{f}_{i+e\mathbb{Z},e} \cdot\lambda=0$.
\item if $\lambda$ has a good addable $i$-node $\gamma$ then  $\widetilde{f}_{i+e\mathbb{Z},e} \cdot\lambda=\mu$ where $[\mu]=[\lambda]\sqcup \{\gamma\}$.
\item If $\lambda$ has no removable  $i$-node then $\widetilde{e}_{i+e\mathbb{Z},e} \cdot\lambda=0$.
\item if $\lambda$ has a good removable $i$-node $\gamma$ then  $\widetilde{e}_{i+e\mathbb{Z},e} \cdot\lambda=\mu$ where $[\mu]=[\lambda]\setminus \{\gamma\}$.
\end{itemize}
Using these operators one can construct the $\widehat{\mathfrak{sl}}_e$-{\it crystal graph} of $\mathcal{F}$, which is the graph with 
\begin{itemize}
 \item vertices: all the partitions $\lambda$ of $n\in \mathbb{N}$,
 \item arrows: there is an arrow from $\lambda$ to $\mu$ colored by $i\in \mathbb{Z}/e\mathbb{Z}$ 
  if and only if  $\widetilde{f}_{i+e\mathbb{Z},e} \cdot \lambda=\mu$, or equivalently if and only if
  $\lambda=\widetilde{e}_{i+e\mathbb{Z},e} \cdot \mu$.
\end{itemize} 
Note that the definition makes sense for $e=\infty$. The corresponding graph, ${\mathfrak{sl}}_\infty$-crystal graph, coincides with the Young graph, which describes   the branching graph of the complex irreducible representations of symmetric groups.

\subsection{Mullineux involution} 
We can first give an interpretation of the set of $e$-regular partitions using Kashiwara operators. The following result can be found for example in \cite[\S2.2]{LLT}.
\begin{Prop}\label{reg}
 A partition $\lambda$ is an $e$-regular partition of $n$  if and only if there exists:
 $(i_1,\ldots,i_n)\in \mathbb{Z}^n$  such that:
 $$\widetilde{f}_{i_1+e\mathbb{Z},e}\cdots \widetilde{f}_{i_n+e\mathbb{Z},e}\cdot \emptyset=\lambda.$$
\end{Prop} 
 
In other words, the vertices in the connected component of the $\widehat{\mathfrak{sl}}_e$-crystal graph containing the empty partition are exactly the $e$-regular partitions. We thus have a subgraph of this crystal graph with vertices all these  $e$-regular partitions.

Recall the definition of the Mullineux involution given in the introduction. 
The following result permits to compute it in a purely combinatorial way thanks to the above results. 
\begin{Th}[Kleshchev]
 Let $\lambda$ be a $e$-regular partition. 
Then, there exists $(i_1,\ldots,i_n)\in \mathbb{Z}^n$ such that:
$$\widetilde{f}_{i_1+e\mathbb{Z},e} \ldots \widetilde{f}_{i_n+e\mathbb{Z},e}. \emptyset =\lambda.$$
Then, there exists an $e$-regular partition  $\mu$ such that:
$$\widetilde{f}_{-i_1+e\mathbb{Z}e} \ldots \widetilde{f}_{-i_n+e\mathbb{Z},e}. \emptyset =\mu.$$
Moreover, we have $m_e (\lambda)=\mu$ where $m_e$ is the Mullineux involution defined in the introduction.
\end{Th}
If $\lambda$ is a partition, every node of its Young diagram has an associated hook, defined as  the set of nodes directly below or to its right (including itself). A partition is called 
 an {\it $e$-core} if it has no hook with $k.e$ nodes for every $k\in \mathbb{N}$.  Of course, if $e$ is sufficiently large comparing to $n$ ($e>n$), every partition  of $n$ is an $e$-core.  
If $\lambda$ is an $e$-core, it is already contained in Mullineux's original paper \cite{Mu} that  $m_e (\lambda)$ is the conjugate partition of $\lambda$ (defined as the partition obtained  by interchanging rows and columns in the Young
diagram of $\lambda$)

\begin{exa}
Let $e=3$ and let $\lambda=(5,2,1,1)$. This is a $3$-regular partition. Then we have:
 $$\widetilde{f}^2_{0+3\mathbb{Z},3}\widetilde{f}^2_{1+3\mathbb{Z},3}\widetilde{f}_{0+3\mathbb{Z},3}\widetilde{f}^2_{2+3\mathbb{Z},3}\widetilde{f}_{1+3\mathbb{Z},3} \widetilde{f}_{0+3\mathbb{Z},3}\emptyset=\lambda .$$
 We get
 $$\widetilde{f}^2_{0+3\mathbb{Z},3}\widetilde{f}^2_{2+3\mathbb{Z},3}\widetilde{f}_{0+3\mathbb{Z},3}\widetilde{f}^2_{1+3\mathbb{Z},3}\widetilde{f}_{2+3\mathbb{Z},3} \widetilde{f}_{0+3\mathbb{Z},3}\emptyset=(4,2,2,1) .$$
 and thus  $m_3 (5,2,1,1)=(4,2,2,1)$.  If $e=6$ then $\lambda=(5,2,1,1)$ is a $6$-core and we have:
 $$\widetilde{f}_{0+6\mathbb{Z},6}\widetilde{f}_{3+6\mathbb{Z},6}\widetilde{f}^2_{4+6\mathbb{Z},6}\widetilde{f}_{3+6\mathbb{Z},6}\widetilde{f}_{2+6\mathbb{Z},6}\widetilde{f}_{1+6\mathbb{Z},6}\widetilde{f}_{5+6\mathbb{Z},6} \widetilde{f}_{0+6\mathbb{Z},6}\emptyset=\lambda .$$
We obtain
  $$\widetilde{f}_{0+6\mathbb{Z},6}\widetilde{f}_{3+6\mathbb{Z},6}\widetilde{f}^2_{2+6\mathbb{Z},6}\widetilde{f}_{3+6\mathbb{Z},6}\widetilde{f}_{4+6\mathbb{Z},6}\widetilde{f}_{5+6\mathbb{Z},6}\widetilde{f}_{1+6\mathbb{Z},6} \widetilde{f}_{0+6\mathbb{Z},6}\emptyset=(4,2,1,1,1),$$
  which is the conjugate partition of $\lambda$, as expected.
\end{exa}

In the following, we will study another way to compute this map without any use of the crystal and the Kashiwara operators. 

\section{Crystal isomorphisms for bipartitions}
In this section, we quickly summarize the needed results to expose our algorithm. These results mainly concern certain expansions of the above discussion to the case of bipartitions.

\subsection{Level $2$ Fock space}
From now we fix  a {\it  bicharge}, that is a couple ${\bf s}=(s_1,s_2)\in \mathbb{Z}^2$. Let us denote by $\Pi^2 (n)$ the set of pairs of partitions ({\it bipartitions}) $(\lambda^1,\lambda^2)$ such that $|\lambda^1|+|\lambda^2|=n$.
One can define the level $2$-Fock space as the $\mathbb{C}$-vector space  with basis indexed by all the elements of $\Pi^2 (n)$
 for $n\in \mathbb{Z}_{\geq 0}$. There is also a notion of crystal for this $2$-Fock space  with similar notions of Kashiwara operators 
  $\widetilde{f}_{i+e\mathbb{Z},e}^{\bf s}$ and   $\widetilde{e}_{i+e\mathbb{Z},e}^{\bf s}$. Importantly, the action of these operators  on each bipartition really depends on the choice of ${\bf s}$.

    To each $\ulambda:=(\lambda^1,\lambda^2)\in \Pi^2 (n)$ is associated  its {\it   Young diagram}:
$$[\ulambda]=\{(a,b,c)\ |  \ a\geq 1,\ c\in \{1,2\}, 1\leq b\leq \lambda_a^c\}.$$
We define the {\it content} of a node  $\gamma=(a,b,c)\in [\ulambda]$ as follows: 
$$\text{cont}(\gamma)=b-a+s_c,$$
 and the residue $\mathrm{res}(\gamma )$ is by definition the content of the node taken modulo $e$. 
 We will say that $\gamma $ is an $i+e\mathbb{Z}$-node of ${\boldsymbol{\lambda}}$ when $%
\mathrm{res}(\gamma )\equiv i +e\mathbb{Z}$ (we will sometimes simply called it an $i$-node). Finally, we say that $\gamma $
is {\it removable} when $\gamma =(a,b,c)\in [{\boldsymbol{\lambda}}]$ and $[{%
\boldsymbol{\lambda}}]\backslash \{\gamma \}$ is the Young diagram of  a bipartition. Similarly, $%
\gamma $ is {\it addable} when $\gamma =(a,b,c)\notin [{\boldsymbol{\lambda}}]$ and $%
[{\boldsymbol{\lambda}}]\cup \{\gamma \}$ is the Young diagram of  a bipartition.

 Let $\gamma$, $\gamma'$ be two removable or addable $i$-nodes of  $\ulambda$.  We  
denote
$$\gamma\prec_{\bf s}\gamma' \qquad 
\stackrel{\text{def}}{\Longleftrightarrow}\qquad \left\{\begin{array}{ll}
\mbox{either} & b-a+s_c<b'-a'+s_{c'},\\
\mbox{or} & b-a+s_c=b'-a'+s_{c'} \text{ and }c>c'.\end{array}\right.$$

For $\ulambda$ a bipartition and $i\in \mathbb{Z}/e\mathbb{Z}$, we can consider its set of addable and removable $i$-nodes.  Let $w^{(e,{\bf s})}_{i}(\ulambda)$ be the word obtained first by writing the
addable and removable $i$-nodes of ${\boldsymbol{\lambda}}$ in {increasing}
order with respect to $\prec _{{\mathbf{s}}}$,  
next by encoding each addable $i$-node by the letter $A$ and each removable $%
i$-node by the letter $R$.\ Write $\widetilde{w}^{(e,{\bf s})}_{i}(\ulambda)=A^{p}R^{q}$ for the
word derived from $w^{(e,{\bf s})}_{i}(\ulambda)$ by deleting as many of the factors $RA$ as
possible. In the following, we will sometimes write 
   $\widetilde{w}_{i}(\ulambda)$ and $w_i (\ulambda)$ instead of    $\widetilde{w}^{(e,{\bf s})}_{i}(\ulambda)$ and $w^{(e,{\bf s})}_{i}(\ulambda)$ if there is no possible confusion.

If $p>0,$ let $\gamma $ be the rightmost addable $i$-node in $%
\widetilde{w}_{i}$. The node $%
\gamma $ is called the {\it good addable $i$-node}. If $r>0$, the leftmost removable $i$-node in 
 $\widetilde{w}_{i}$ is called the {\it good removable $i$-node}. The definition of the Kashiwara 
  operators 
  $\widetilde{f}_{i+e\mathbb{Z},e}^{\bf s}$ and   $\widetilde{e}_{i+e\mathbb{Z},e}^{\bf s}$ follows then exactly as in \S \ref{l1}.
  In the same spirit as in the above discussion, one can also define a certain subset of bipartitions $\Phi^{({\bf s},e)} (n)$:
\begin{Def} We say that $(\lambda^1,\lambda^2)$ is an {\it Uglov bipartition} associated with ${\bf s}\in \mathbb{Z}^2$ if there exist 
$(i_1,\ldots,i_n)\in \mathbb{Z}^n$  such that:
$$\widetilde{f}^{\bf s}_{i_1+e\mathbb{Z},e} \ldots \widetilde{f}^{\bf s}_{i_n+e\mathbb{Z},e}. (\emptyset,\emptyset) =(\lambda^1,\lambda^2).$$
We denote by $\Phi_{(e,{\bf s })}$ the set of Uglov bipartitions and by $\Phi_{(e,{\bf s })}(n)$ the set   $\Phi_{(e,{\bf s })}\cap \Pi^2 (n)$.
\end{Def}

We make the three important   following remarks. 
\begin{Rem}\label{three}
\begin{enumerate}
\item Assume that $k\in \mathbb{Z}$ then there is a unique bijection:
$$\psi_{(e,(s_1,s_2+ke))} : \Phi_{(e,(s_1,s_2+ke))} \to  \Phi_{(e,(s_1,s_2+(k+1)e))},$$
preserving the rank of bipartitions and commuting with the Kashiwara operators, that is, for all $i\in \mathbb{Z}$ and $\lambda \in \Phi_{(e,(s_1,s_2+ke))}$, we have
$$\psi_{(e,(s_1,s_2+ke))} (\widetilde{f}^{(s_1,s_2+ke)}_{i+e\mathbb{Z},e} .\lambda)= \widetilde{f}^{(s_1,s_2+(k+1)e)}_{i+e\mathbb{Z},e} . \psi_{(e,(s_1,s_2+ke))}( \lambda),$$
and 
$$\psi_{(e,(s_1,s_2+ke))} (\widetilde{e}^{(s_1,s_2+ke)}_{i+e\mathbb{Z},e} .\lambda)= \widetilde{e}^{(s_1,s_2+(k+1)e)}_{i+e\mathbb{Z},e} . \psi_{(e,(s_1,s_2+ke))}( \lambda).$$
This bijection may be computed thanks  to a purely combinatorial algorithm given in section \S \ref{comb}. This map is called a {\it crystal isomorphism}. 
\item By \cite[\S 6.2.16]{GJ}, in the case where $|s_2-s_1|>n-1-e$, the bijection
$\psi_{(e,(s_1,s_2))}$ restricted to $\Phi_{(e,(s_1,s_2))} (n)$
 is always the identity.  We say that $(s_1,s_2)$ is {\it very dominant} (comparing to $n$). 
  This implies in particular that 
as soon as  $|s_2-s_1|>n-1-e$, the set   $\Phi_{(e,{\bf s })} (n)$ only depends on the congruence class of $(s_1,s_2)$  modulo $e$ (and not on $k$).
Similarly, the action of the Kashiwara operators on the bipartitions  of rank less than $n$ does not depend on $k$ if the above condition is satisfied. 
 The set is then called the set of Kleshchev  bipartitions. The set of Kleshchev bipartitions of rank $n$  will be denoted  by $\Phi_{(e,{\bf s })}^K(n)$
  and  we denote $\Phi_{(e,{\bf s })}^K:=\sqcup_{n\geq 0} \Phi_{(e,{\bf s })}^K(n)$
 \item One can  define a bijection:
  $$\widetilde{\psi}_{(e,(s_1,s_2))}: \Phi_{(e,(s_1,s_2))}  \to \Phi_{(e,{\bf s })}^K,$$
  as follows. Let $n\in \mathbb{Z}_{\geq 0}$ and let $\ulambda=(\lambda^1,\lambda^2)\in \Phi_{(e,(s_1,s_2))}$. 
 Assume  that  $k\in \mathbb{Z}_{>0}$ is such that a  $|s_2+ke-s_1|>n-1-e$, then we define:
 $$\widetilde{\psi}_{(e,(s_1,s_2))} (\lambda^1,\lambda^2):=\psi_{(e,(s_1,s_2+(k-1)e))}\circ \ldots \circ \psi_{(e,(s_1,s_2+e))}\circ \psi_{(e,(s_1,s_2))}(\lambda^1,\lambda^2).$$
 Due to the above remark, this bijection does not depend on $k$. 
\end{enumerate}
\end{Rem}

\subsection{Mullineux map}

There exists a Mullineux type map in the case of bipartitions.   Let ${\bf s}=(s_1,s_2)\in \mathbb{Z}^2$ and let 
$-{\bf s}:=(-s_1,-s_2)$. 
Our Mullineux map will be a map:
  $$\mathcal{M}_{(e,{\bf s})}: \Phi^K_{(e,{\bf s})} \to \Phi^K_{(e,-{\bf s})},$$
which is uniquely defined as follows. Let $\ulambda \in \Phi^K_{(e,{\bf s})} (n)$.
 Let $n\in \mathbb{Z}_{>0}$. Let ${\bf s}^1=(s_1,s_2+ke)$ be a very dominant bicharge such that 
 ${\bf s}^1\equiv {\bf s} +e\mathbb{Z}$ 
 and let   ${\bf s}^2$
be a very dominant bicharge such that  ${\bf s}^2\equiv -{\bf s} +e\mathbb{Z}$. There exists
 $(i_1,\ldots,i_n)\in \mathbb{Z}^n$ such that:
$$\widetilde{f}^{{\bf s}^1}_{i_1+e\mathbb{Z},e} \ldots \widetilde{f}^{{\bf s}^1}_{i_n+e\mathbb{Z},e}. \uemptyset =\ulambda.$$
Then it is shown in \cite[\S 2]{F2} that there exists $\umu \in \Phi_{(e,{\bf s }^2)}^K$ such that:
 $$\widetilde{f}^{{\bf s}^2}_{-i_1+e\mathbb{Z},e} \ldots \widetilde{f}^{{\bf s}^2}_{-i_n+e\mathbb{Z},e}. \uemptyset =\umu.$$
 We denote $\mathcal{M}_{(e,(s_1,s_2))} (\ulambda):=\umu$. 
 Then it is shown in \cite[Prop. 4.2]{JL0} that $\umu=(m_e (\lambda^1),m_e (\lambda^2))$.   
In the following section, we will use this property to deduce our conjectural algorithm.

\section{Explicit computations and a combinatorial property}\label{comb}
In this section, we explain how one can compute the above crystal isomorphisms. 
Our main conjecture is relied on a combinatorial conjectural property of  these maps. This property can in fact be settled in a completely 
general  framework.

\subsection{A combinatorial map}\label{algo} We recall here results from \cite{JL}. 
Let $e$ be a positive integer. For $r$ a positive integer, we denote by $\mathcal{P}^{r}$ the set of strictly increasing partitions in $r$ parts. 
 Let $m_1$ and $m_2$ be two integers such that $m_1 \leq m_2$. 

Let $(X_1,X_2)\in \mathcal{P}^{m_1} \times \mathcal{P}^{m_2}$.  Set 
 $$X_1=(a_{1},\ldots,a_{m_1}),\ X_2=(b_{1},\ldots,b_{m_2}).$$ 
 We define an injection 
$\varphi : X_1 \to X_2$ as follows.
\begin{itemize}
\item We set 
$$\varphi (a_1) =\operatorname{max} \{ b_j \ |\ j=1,\ldots,m_2, b_j \leq a_1\},$$
if it exists. Otherwise, we set 
$$\varphi (a_1)=\operatorname{max} \{ b_j \ |\ j=1,\ldots,m_2\}.$$
\item We repeat this procedure with $(a_2,\ldots,a_{m_1})$ and $X_2 \setminus \{\varphi (a_1)\}$ and thus associate 
to each element of $X_1$ a unique element in $X_2$. 
\end{itemize}
We now define a map:
 $$\Psi_{(e,(m_1,m_2))}  : \mathcal{P}^{m_1}\times \mathcal{P}^{m_2} \to  \mathcal{P}^{m_1}\times \mathcal{P}^{m_2+e},$$
with  $(Y_1,Y_2):=\Psi_{(e,(m_1,m_2))} (X_1,X_2)$:
$$Y_1=\{ \varphi (a_j) \ |\ j=1,\ldots,m_1\},$$
$$Y_2=\{ a_j+e\ |\ j=1,\ldots,m_1\} \cup\{ b_j +e \ |\ j=1,\ldots,m_2; b_j \notin Y_1 \} \cup \{0,1\ldots,e-1\},$$
where we reorder these two sets so that $Y_1 \in \mathcal{P}^{m_1}$ and $Y_2 \in \mathcal{P}^{m_2+e}$. 

\begin{Rem}
The map is bijective and  $(\Psi_{(e,(m_1,m_2))})^{-1}$ may be computed as follows. Assume that  $(Y_1,Y_2):=\Psi_{(e,(m_1,m_2))} (X_1,X_2)$  then take 
$Y_2'$ be the set $\{y-e \ |\ y\in Y_2\setminus \{0,1,\ldots,e-1\}\}$. Then 
we define an injection 
$\varphi ': Y_1 \to Y_2$ as follows.
\begin{itemize}
\item We set 
$$\varphi '(a_1) =\operatorname{min} \{ b_j \ |\ j=1,\ldots,m_2, b_j \geq a_1\},$$
if it exists. Otherwise, we set 
$$\varphi' (a_1)=\operatorname{min} \{ b_j \ |\ j=1,\ldots,m_2\}.$$
\item We repeat this procedure with $(a_2,\ldots,a_{m_1})$ and $X_2 \setminus \{\varphi '(a_1)\}$ and thus associate 
to each element of $Y_1$ a unique element in $Y_2'$.  Then we have
$$X_1=\{ \varphi' (a_j) \ |\ j=1,\ldots,m_1\},$$
$$X_2=\{ a_j\ |\ j=1,\ldots,m_1\} \cup\{ b_j \ |\ j=1,\ldots,m_2; b_j \notin X_1\}.$$
(after reordering the elements)
\end{itemize}

\end{Rem}
\begin{Rem}\label{ide}
In the case where $X_1 \subset X_2$, it follows from the above definition that 
$$\Psi_{(e,(m_1,m_2))}(X_1,X_2)=(X_1,X_2+e).$$
\end{Rem}

\subsection{Connection with crystal isomorphisms}
Assume that   ${\bf s}=(s_1,s_2)\in \mathbb{Z}^2$ and assume in addition that $s_1\leq s_2$ (we only need this case in the following but note that there is an analogue description of the crystal isomorphisms if $s_1\geq s_2$, see \cite{JL}). Let $\ulambda=(\lambda^1,\lambda^2)$ be a bipartition 
 of $n$ in  $\Phi_{(e,(s_1,s_2))}$.   One can assume that there exists an integer $m$ such that $\lambda^1=(\lambda^1_{1},\ldots,\lambda^1_{m+s_1})$ and $\lambda^2=(\lambda^2_{1},\ldots,\lambda^2_{m+s_2})$,
  adding parts equal to $0$ if necessary. For $i=1,\ldots,m+s_1$, we set
  $$\beta^1_j=\lambda^1_j-j+s_1+m.$$
  For $i=1,\ldots,m+s_2$, we set
  $$\beta^2_j=\lambda^2_j-j+s_2+m.$$
We then define   $X^{s_1,m}_1 (\lambda^1):=(\beta^1_{s_1+m},\ldots,\beta_2^1)\in \mathcal{P}^{m+s_1}$ and $X^{s_2,m}_2 (\lambda^2):=(\beta^2_{s_2+m},\ldots,\beta_1^2)\in \mathcal{P}^{m+s_2}$.

By \cite{JL}, we get:
\begin{Prop}
Keeping the above notations, 
We have $$\psi_{(e,(s_1,s_2))} (\lambda^1,\lambda^2)=(\mu^1,\mu^2),$$
where $(\mu^1,\mu^2)$ is the unique bipartition of $n$ such that 
$$\Psi_{(e,(s_1+m,s_2+m))}  (X^{s_1,m}_1 (\lambda^1),X^{s_2,m}_2 (\lambda^2))=(X^{s_1,m}_1 (\mu^1),X^{s_2+e,m}_2 (\mu^2)).$$
\end{Prop}
\begin{Rem}\label{ide2}
In the case where $X^{s_1,m}_1 (\lambda^1) \subset X^{s_2,m}_2 (\lambda^2)$, by Remark \ref{ide}, we obtain:
$$\psi_{(e,(s_1,s_2))} (\lambda^1,\lambda^2)=(\lambda^1,\lambda^2).$$
\end{Rem}

\subsection{Computing the map $\widetilde{\psi}_{(e,(s_1,s_2))}$}\label{asyp}. Assume that $s_1\leq s_2$. 
 To compute 
 $\widetilde{\psi}_{(e,(s_1,s_2))}$, as explained in Remark \ref{three} $(3)$,  we have to fix $n\in \mathbb{Z}_{\geq 0}$ and compute 
 $\widetilde{\psi}_{(e,(s_1,s_2))}|_{\Phi_{(e,{\bf s })} (n)  }$. If 
   $k\in \mathbb{Z}_{>0}$  is  such that  $|s_2+ke-s_1|>n-1-e$,  we have  to compose $k$ crystal isomorphisms: 
 $$ \widetilde{\psi}_{(e,(s_1,s_2))}|_{\Phi_{(e,{\bf s })} (n)  } :=\psi_{(e,(s_1,s_2+(k-1)e))}\circ \cdots \circ \psi_{(e,(s_1,s_2+e))}\circ \psi_{(e,(s_1,s_2))}|_{\Phi_{(e,{\bf s })} (n) }$$
However, in most of the cases,  if we want to compute the image of a particular bipartition $\ulambda\in \Phi_{(e,{\bf s })} (n) $ under  $\widetilde{\psi}_{(e,(s_1,s_2))}$ 
one can be considerably more efficient thanks to the following remark. Let $\ulambda\in \Phi_{(e,{\bf s })} (n) $   and  $h:=\text{max}\{i\in \mathbb{Z}_{>0} \ |\ \lambda^2_i\neq 0\}+1$. 
Assume that 
\begin{equation}\label{simp} \lambda^1_1-1+s_1\leq s_2-h,\end{equation}
then we have for all relevant $m$, and for all $k\geq 0$ : $X^{s_1,m}_1 (\lambda^1)\subset X^{s_2,m}_2 (\lambda^2)$.
By Remarks \ref{ide} and \ref{ide2},  this implies that 
$$\psi_{(e,(s_1,s_2))} (\lambda^1,\lambda^2)=(\lambda^1,\lambda^2).$$
But now we also have $\lambda^1_1-1+s_1\leq s_2+e-h$ and thus we obtain 
$$\psi_{(e,(s_1,s_2+e))} (\lambda^1,\lambda^2)=(\lambda^1,\lambda^2).$$
By an immediate induction, we deduce that   for all $k\geq 0$ , we have:
$$\psi_{(e,(s_1,s_2+ke))} (\lambda^1,\lambda^2)=(\lambda^1,\lambda^2).$$
In this case, we thus simply have:
$$\widetilde{\psi}_{(e,(s_1,s_2))}(\lambda^1,\lambda^2)=(\lambda^1,\lambda^2).$$
Of course, a similar result holds for $(\widetilde{\psi}_{(e,(s_1,s_2))})^{-1}$ : if $(\lambda^1,\lambda^2)\in  \Phi^K_{(e,{\bf s})} $ satisfies the above property, then we have 
 for all $k\geq 0$ that $(\lambda^1,\lambda^2)\in   \Phi_{(e,(s_1,s_2+ke))}$  and 
 $(\psi_{(e,(s_1,s_2+ke))} )^{-1}(\lambda^1,\lambda^2)=(\lambda^1,\lambda^2)$. 

We end this section with our combinatorial conjecture 
\subsection{A combinatorial conjecture}
Our main conjecture is the following one:
\begin{conj}\label{conj}
Let  $X\in \mathcal{P}_{m}$ and let $k\in \mathbb{N}$, set 
$$(X_1,X_2)=\Psi_{(e,(0,ke))} \circ  \ldots \circ \Psi_{(e,(0,e))} \circ  \Psi_{(e,(0,0))} ( X,X)\in \mathcal{P}^{m} \times \mathcal{P}^{m+ke}.$$
Then if $k$ is odd, we have  $X_1 \subset X_2$. 
 \end{conj}

We prove the conjecture in the case $k=1$.  Note that if $X_1 \subset X_2$ then $\varphi$ is the identity. We thus have that 
 $$\Psi_{(e,(0,0))} (X,X)=(X,X+e \cup\{0,1,\ldots,e-1\}).$$
Now se set 
 $$\Psi_{(e,(0,e))} (X,X+e \cup\{0,1,\ldots,e-1\})=(Y_1,Y_2).$$
 From the above procedure, the elements of  $Y_1$ are some elements of $X+e \cup\{0,1,\ldots,e-1\} $  and $Y_2$ is given by 
  $\{0,1,\ldots,e-1\}$ together with all the elements of $X+e$ and other elements of $X+e \cup\{0,1,\ldots,e-1\} $  translated by $e$. We thus have 
   $Y_1 \subset Y_2$.

In the following, it will be convenient to write the image of an element $(X_1,X_2)\in \mathcal{P}_{m_1} \times \mathcal{P}_{m_2}$ under 
 a map $\Psi_{(e,{\bf s})}$ as $\left( \begin{array}{c}  Y_1 \\ Y_2 \end{array} \right)$ instead of $(Y_1,Y_2)$.  This is what we are going to do in the following example. Assume that 
$$X=\{0,3,5,6,10,12,18,20\},$$
and $e=3$. 
We check that 
$$\Psi_{(3,(0,0))} (X,X)=\left(
\begin{array}{cccccccccccc}
0&1&2&3&6&8&9&13&15&18&21&23\\
0&3&5&6&10&12&15&18&20\\
\end{array}\right)
$$
$$\Psi_{(3,(0,3))}\circ \Psi_{(3,(0,0))} (X,X)
=\left(
\begin{array}{ccccccccccccccc}
0&1&2&3&4&6&8&9&13&15&18&21&23&24&26 \\
0&2&3&6&8&9&13&15&18
\end{array}\right)$$
and the  set below  is included in the set above, as claimed by the conjecture. Then by applying $\Psi_{(3,(0,6))}$ we get 
$$\left(
\begin{array}{cccccccccccccccccc}
0&1&2&3&4&5&6&7&9&11&12&16&18&21&24&26&27&29 \\
0&2&3&6&8&9&13&15&18
\end{array}\right)$$
and the action of $\Psi_{(3,(0,9))}$ then gives 
$$\left(
\begin{array}{cccccccccccccccccc}
0&1&2&3&4&5&6&7&8&9&11&12&16&18&... \\
0&2&3&6&7&9&11&12&18
\end{array}\right)$$
which yet satisfies the inclusion property. 

Note that in the assumptions of the conjecture, we really need $k$ to be odd. In the case when $k$ is even, the assertion  is wrong as we can see in the above example. 

\begin{Rem}
This conjecture has been checked for all couples $(X,X)=(X^{0,m}(\lambda),X^{0,m}(\lambda))$ with $\lambda$ 
an arbitrary partition of rank $n$ with 
 $n\leq 40$ (and  $e$ arbitrary). A proof for the conjecture has  already been obtained by M.Fayers when $e=2$ \cite{mat}. \end{Rem}

\section{Conjectural consequences on crystal isomorphisms}

We first establish some elementary results concerning $e$-regular partitions and then explain our conjectural algorithm.

\begin{Prop}\label{2ereg}
Let $\lambda$ be an $e$-regular partition and consider a sequence $(i_1,\ldots,i_n)\in \mathbb{Z}^n$ such that:
$$\widetilde{f}_{i_1+\mathbb{Z}e,e} \ldots \widetilde{f}_{i_n+\mathbb{Z}e,e}. \emptyset =\lambda.$$
 Then  we have 
  $$(\widetilde{f}^{(0,0)}_{i_1+\mathbb{Z}e,e})^2 \ldots (\widetilde{f}^{(0,0)}_{i_n+\mathbb{Z}e,e})^2 (\emptyset,\emptyset)=(\lambda,\lambda),$$
and in particular we have $(\lambda,\lambda)\in \Phi_{(e,(0,0))}$.
\end{Prop}
\begin{proof}
Let $\lambda\in \operatorname{Reg}_e (n)$. By Proposition \ref{reg}, there exists $(i_1,\ldots,i_n)\in \mathbb{Z}^n$ such that:
$$\widetilde{f}_{i_1+\mathbb{Z}e,e} \ldots \widetilde{f}_{i_n+\mathbb{Z}e,e}. \emptyset =\lambda.$$
We set 
$$\widetilde{\lambda}:=\widetilde{f}_{i_1+\mathbb{Z}e,e} \ldots \widetilde{f}_{i_{n-1}+\mathbb{Z}e,e}. \emptyset.$$
By induction, we have that 
  $$(\widetilde{f}^{(0,0)}_{i_1+\mathbb{Z}e,e})^2 \ldots (\widetilde{f}^{(0,0)}_{i_{n-1}+\mathbb{Z}e,e})^2 (\emptyset,\emptyset)=(\widetilde{\lambda},\widetilde{\lambda}).$$
Assume that  
   $$w_{i_n+e\mathbb{Z}} (\widetilde{\lambda})=Z_1,\ldots Z_m.$$
where  for all $i=1,\ldots,m$, $Z_i\in \{A,R\}$ correspond to a node $(a_i,b_i)$. Then we have:
    $$w_{i_n+e\mathbb{Z}} (\widetilde{\lambda},\widetilde{\lambda})=T_1,\ldots, T_{2m},$$
  where $T_{2i-1}=Z_i$ correspond to the node $(a_i,b_i,2)$ for $i=1,\ldots,m$ and  $T_{2i}=Z_i$ 
   for $i=1,\ldots,m$ corresponds to the node $(a_i,b_i,1)$. 
   It follows that if $(a_k,b_k)$ is a good addable $i_n+e\mathbb{Z}$-node  for $\widetilde{\lambda}$ then 
    $(a_i,b_i,2)$  is a good addable $i_n+e\mathbb{Z}$-node  for $(\widetilde{\lambda},\widetilde{\lambda})$ 
   and  $(a_i,b_i,1)$  is a good addable $i_n+e\mathbb{Z}$-node  for $(\widetilde{\lambda},\lambda)$. 
    We conclude that 
    $$(\widetilde{f}^{(0,0)}_{i_1+\mathbb{Z}e,e})^2 \ldots (\widetilde{f}^{(0,0)}_{i_n+\mathbb{Z}e,e})^2 (\emptyset,\emptyset)=(\lambda,\lambda)$$
    as required. 
\end{proof}
The following result comes from \cite[Lemma 3.2.12]{J} (see also \cite{Hu} for an similar result, but for a different realization of the Fock space).
\begin{Prop}\label{2ereg2}
Let $\lambda$ be an $e$-regular partition and let $(i_1,\ldots,i_n)\in \mathbb{Z}^n$ be such that:
$$\widetilde{f}_{i_1+\mathbb{Z}e,e} \ldots \widetilde{f}_{i_n+\mathbb{Z}e,e}. \emptyset =\lambda.$$
 Then  we have:
  $$\widetilde{f}^{(0,e)}_{i_1+2\mathbb{Z}e,2e} \widetilde{f}^{(0,e)}_{i_1+e+2\mathbb{Z}e,2e}   \ldots \widetilde{f}^{(0,e)}_{i_n+2\mathbb{Z}e,2e}  \widetilde{f}^{(0,e)}_{i_n+e+2\mathbb{Z}e,2e}   (\emptyset,\emptyset)=(\lambda,\lambda),$$
and in particular we have $(\lambda,\lambda)\in \Phi_{(2e,(0,e))}$.
\end{Prop}

We now use the above proposition together with the  following result which  rephrases Conjecture \ref{conj} in terms of crystal isomorphisms.
 
 Our algorithm is now built in the following result which assumes Conjecture \ref{conj}. 
 \begin{Prop}\label{isop}
 Assume that Conjecture \ref{conj} is true then for all $e$-regular partitions $\lambda$ of $n$ and $k\in \mathbb{Z}_{>0}$, we have:
 $$\psi_{(2e,(0,(2k+1)e))}\circ  \cdots  \circ \psi_{(2e,(0,3e))}  \circ  \psi_{(2e,(0,e))}  (\lambda,\lambda)= \psi_{(e,(0,(2k+1)e))}  \circ \cdots \circ \psi_{(e,(0,e))}  \circ   \psi_{(e,(0,0))} (\lambda,\lambda).$$
 In particular, we have 
  $$ \widetilde{\psi}_{(2e,(0,e))} (\lambda,\lambda)= \widetilde{\psi}_{(e,(0,0))}  (\lambda,\lambda).$$
 \end{Prop}
\begin{proof}
Let $\lambda$ be an $e$-regular partition of $n$.  By Prop. \ref{2ereg} and \ref{2ereg2}, we have that $(\lambda,\lambda)\in \Phi_{(2e,(0,e))} \cap \Phi_{(e,(0,0))}$. 
By Remark \ref{ide}, we have 
$$ \psi_{(e,(0,0))} (\lambda,\lambda)=(\lambda,\lambda)\in \Phi_{(e,(0,e))}. $$
We thus have $(\lambda,\lambda)\in \Phi_{(e,(0,e))}\cap  \Phi_{(2e,(0,e))}$. 
Now, the algorithm to compute the image of a bipartiton under the crystal isomorphism $ \psi_{(e,(s_1,s_2))}$  does not depend on $e$. 
 This implies that:
$$ \psi_{(e,(0,e))}  (\lambda,\lambda)= \psi_{(2e,(0,e))}  (\lambda,\lambda).$$
We then argue by induction. Assume that 
 $$(\mu^1,\mu^2):=\psi_{(2e,(0,(2k-1)e))}\circ  \cdots  \circ \psi_{(2e,(0,3e))}  \circ  \psi_{(2e,(0,e))}  (\lambda,\lambda)= \psi_{(e,(0,(2k-1)e))}  \circ \cdots \circ \psi_{(e,(0,e))}  \circ   \psi_{(e,(0,0))} (\lambda,\lambda).$$
We use Conjecture \ref{conj} and Remark \ref{ide} to deduce that 
 $$\psi_{(e,(0,2ke))}  (\mu^1,\mu^2)=(\mu^1,\mu^2).$$
 Again, the algorithm  to compute the crystal isomorphisms implies that:
$$\psi_{(2e,(0,(2k+1)e))} (\mu^1,\mu^2)=\psi_{(e,(0,(2k+1)e))} (\mu^1,\mu^2).$$
and we are done. 
\end{proof}

\begin{Rem}

Note that in fact to prove the above result, we only need to prove Conjecture \ref{conj}  in the case where $X \in  \mathcal{P}_{m}$ is such that there is no $i\in X$ such that 
$i,i+1,\ldots, i+e-1$ are in $X$ (this corresponds to the $\beta$-sets associated with $e$-regular partitions).  We note that Matt Fayers has  a proof of this conjecture in the case $e=2$  and for $2$-regular partitions \cite{mat}.

\end{Rem}

Assuming that the conjecture \ref{conj} is true, we will now be able to obtain our algorithm. To do this, we will use a remarkable property of the Mullineux map which is available when the bicharge is very dominant. We will thus use the map   $\widetilde{\psi}_{(e,{\bf s})}$ which ``take to the very dominant world'' (see Remark \ref{three} (3)).

\begin{Prop}
Assume that Conjecture \ref{conj} is satisfied. Let $\lambda$ be an $e$-regular partition and denote $(\mu^1,\mu^2):=\widetilde{\psi}_{(e,(0,0))}(\lambda,\lambda)$.  We have: $$\begin{array}{rcl}
 \widetilde{\psi}_{(e,(0,0))}(m_e (\lambda),m_e (\lambda))&=&
\widetilde{\psi}_{(2e,(0,e))}(m_e (\lambda),m_e (\lambda)) \\
&=&(m_{2e} (\mu^1),m_{2e} (\mu^2))\\
&=& (m_e (\mu^1),m_e (\mu^2)).  \end{array}$$
\end{Prop}

\begin{proof}

Let $\lambda$ be an $e$-regular partition. There exists $(i_1,\dots,i_n)\in \mathbb{Z}^n$ such that:
$$\widetilde{f}_{i_1+\mathbb{Z}e,e} \ldots \widetilde{f}_{i_n+\mathbb{Z}e,e}. \emptyset =\lambda.$$
 By Proposition \ref{2ereg}, we have that $(\lambda,\lambda)\in \Phi_{(e,(0,0))}$
 and by Proposition \ref{2ereg2}, we have$(\lambda,\lambda)\in \Phi_{(2e,(e,0))}$. Now  By Proposition \ref{isop} , we have that 
$$\widetilde{\psi}_{(e,(0,0))}(\lambda,\lambda)=\widetilde{\psi}_{(2e,(0,e))}(\lambda,\lambda).$$
Set $\umu:=(\mu^1,\mu^2):=\widetilde{\psi}_{(e,(0,0))}(\lambda,\lambda)$. Then by Proposition \ref{2ereg2},  we have
$\mathcal{M}_{(e,(0,0))} (\umu)=(m_e (\mu^1),m_e (\mu^2))$ and
$\mathcal{M}_{(2e,(0,e))} (\umu)=(m_{2e} (\mu^1),m_{2e} (\mu^2))$.

On the other hand take $\lambda':=(m_e (\lambda),m_e (\lambda))$. If we argue exactly as above, we have 
$$\widetilde{\psi}_{(e,(0,0))}(m_e (\lambda),m_e (\lambda))=\widetilde{\psi}_{(2e,(0,e))}(m_e (\lambda),m_e (\lambda)).$$
By definition we have 
 $$(\widetilde{f}^{(0,0)}_{-i_1+\mathbb{Z}e})^2 \ldots (\widetilde{f}^{(0,0)}_{-i_n+\mathbb{Z}e})^2 (\emptyset,\emptyset)=(m_e(\lambda),m_e (\lambda)),$$
 and this thus implies that 
 $$\widetilde{\psi}_{(e,(0,0))}(m_e (\lambda),m_e (\lambda))=\mathcal{M}_{(e,(0,0))} (\umu),$$
 and thus the result follows. 
\end{proof}

\begin{abs}
The algorithm can now be stated as follows.

\begin{enumerate}
\item If $e$ is sufficiently large, we know the Mullineux image of any $e$-regular partition because then any $e$-regular partition is an $e$-core and thus its Mullineux image is its conjugate partition.
\item Assume that we know $m_{2e}$. Let $\lambda$ be an $e$-regular partition. We compute:
$$(\mu^1,\mu^2):=\widetilde{\psi}_{(2e,(0,e))}(\lambda,\lambda).$$
\item Then we compute:
 $$(\nu^1,\nu^2):=(\widetilde{\psi}_{(2e,(0,e))})^{-1}(m_{2e}(\mu^1),m_{2e}(\mu^2)).$$
 \item We must have 
 $$m_e (\lambda)=\nu^1=\nu^2.$$

\end{enumerate} 

 \end{abs}

\subsection{Example}
Take $e=3$ and the $3$-regular partition $\lambda=(6,5,2,2,1,1)$. This is a partition of rank $17$ and so the very dominant case is reached if $s_2-s_1>30$.  
To perform our algorithm, we must compute:
$$\widetilde{\psi}_{(2e,(0,0))} (\lambda,\lambda)= \psi_{(2e,(0,ke))} \circ  \cdots \circ \psi_{(2e,(0,3e))} \circ  \psi_{(2e,(0,e))} (\lambda,\lambda)$$
until we reach the "very dominant case". We consider the $\beta$-sets associated with the bipartition $(\lambda,\lambda)$ with respect to the bicharge $(0,3)$:
$$\left(
\begin{array}{ccccccccc}
0 & 1 & 2 & 4 & 5 & 7 & 8 & 12& 14 \\
1 & 2 & 4 & 5 & 9 & 11
\end{array}\right)$$
To compute   $ \psi_{(2e,(0,e))} (\lambda,\lambda)$, we need to apply the algorithm described in \S \ref{algo}. We obtain:
$$\left(
\begin{array}{ccccccccccccccc}
0 &1 & 2 & 3 & 4 & 5 & 6 & 7 & 8 & 11 & 12 & 15 & 17 & 18& 20 \\
1 & 2 & 4 & 5 & 7 & 8
\end{array}\right)$$
 The associated bipartition is $((3,3,2,2,1,1),(6,5,5,4,1,1))$.  In principle, we have to apply again the algorithm until the ``very dominant case'', but note that we are already in the case described in \S \ref{asyp} so 
$$\widetilde{\psi}_{(2e,(0,0))} (\lambda,\lambda)=(\mu^1,\mu^2)=((3,3,2,2,1,1),(6,5,5,4,1,1)$$
By induction, we know $m_6 (3,3,2,2,1,1)=(6,4,2)$ (because $(3,3,2,2,1,1)$ is a $6$-core) and $m_6 (6,5,5,4,1,1)=(11,9,2)$. 
So now we have to compute $(\widetilde{\psi}_{(6,(0,0))})^{-1}$ for $((6,4,2),(11,9,2))$ starting from the very dominant case. In fact, using Remark \ref{asyp} again, we see that
$((6,4,2),(11,9,2))$  is in $\Phi_{(6,(0,9))}$ and that 
$$(\widetilde{\psi}_{(6,(0,0))})^{-1} ((6,4,2),(11,9,2))=(\psi_{(6,(0,3))})^{-1}((6,4,2),(11,9,2))$$
To compute this latter expression, we use our (reversed) algorithm, 
we consider the following symbol : 
$$\left(
\begin{array}{cccccccccccc}
 0 & 1 & 2 &5 & 13 & 16 \\
2 & 5 & 8
\end{array}\right)$$
This gives 
$$\left(
\begin{array}{cccccccccccc}
 0 & 1 & 2 &5 & 8 & 16 \\
2 & 5 & 13
\end{array}\right)$$
We get $((11,4,2),(11,4,2))$ and one can check that we indeed have $m_e (\lambda)=(11,4,2)$.


\begin{thebibliography}{99}    



\addcontentsline{toc}{chapter}{Bibliography}
%
%
%
%
%
%
%
%
%
%
%
%
%
%
%
%
%
%
%
%
%
%
%
%
%
%
%
%
%
%
%
%
%
%
%
%
\bibitem{BO} C. Bessenrodt and J.Olsson, On residue symbols and the Mullineux conjecture, J. Algebraic Comb., 7 (1998), 227-251. 

\bibitem{Br} J. Brundan, Modular branching rules and the Mullineux map for Hecke algebras of type A, Proc. London Math. Soc. (3) 77 (1998), no. 3, 551--581. 

\bibitem{BK} J. Brundan and J. Kujawa, A new proof of the Mullineux Conjecture, J. Algebraic Combin. 18 (2003), 13--39.

\bibitem{DiY}
P. Dimakis and G. Yue,
Combinatorial wall-crossing and the Mullineux involution
Journal of Algebraic Combinatorics volume 50, pages 49--72 (2019)

\bibitem{DJ}
O.Dudas and N.Jacon,
 Alvis-Curtis duality for finite general linear groups and a generalized Mullineux involution,
SIGMA, 14, 2018


\bibitem{F} M. Fayers, 
Regularisation, crystals and the Mullineux map, J. Comb. Algebra 6 (2022), 315--352.

\bibitem{F2} M. Fayers, 
Weights of multipartitions and representations of Ariki-Koike algebras II: canonical bases,
	J. Algebra 319 (2008) 2963--2978.
	
	\bibitem{mat}
	M.Fayers,
	Private communication. 
\bibitem{GJ}
M.Geck and  N.Jacon, Representations of Hecke algebras at roots of unity. Algebra and Applications, 
 15. Springer-Verlag London, Ltd., London, 2011. 
\bibitem{GJN}
T. Gerber, N.Jacon and E. Norton, Generalized Mullineux involution and perverse equivalences,
Pacific Journal of Mathematics, Vol. 306 (2020), No. 2, 487--517

\bibitem{FK}
B. Ford, A. Kleshchev, 
A proof of the Mullineux conjecture, Math. Z. 226 (2) (1997) 267--308

\bibitem{Hu}
H. Lin and J. Hu, 
Crystal of affine type $A$
and Hecke algebras at a primitive $2l$th root of unity,
Journal of Algebra
Volume 589, 1 January 2022, Pages 51-81.

\bibitem {JMu} N. Jacon ,
Crystal isomorphisms and Mullineux incolution I, to appear in Combinatorial Theory. 

\bibitem{J}N. Jacon
Two maps on affine type A crystals and Hecke algebras,,
Electronic Journal of Combinatorics, Vol. 28 (2), 2021.


\bibitem {JL} N. Jacon and C. Lecouvey, Crystal isomorphisms for irreducible
highest weight $\mathcal{U}_{v}(\widehat{\mathfrak{sl}}_{e})$-modules of
higher level, Algebras and Representation theory 13, 467-489, 2010.



\bibitem {JL0} N. Jacon and C. Lecouvey,
On the Mullineux involution for Ariki-Koike algebras, Journal of Algebra
Volume 321, Issue 8, 15 April 2009, Pages 2156--2170. 

\bibitem {JL3} N. Jacon and C. Lecouvey,  Kashiwara and Zelevinsky involutions in affine type A,
Pacific Journal of Mathematics, Vol. 243, No. 2, 2009, 287--311.  

\bibitem{K}
A. Kleshchev, 
Branching rules for modular representations of symmetric groups III: Some corollaries and a problem of Mullineux, J. Lond. Math. Soc. 54 (1996) 25--38.

\bibitem{LLT}
A. Lascoux, B. Leclerc, and Jean-Yves Thibon.
\newblock {Hecke algebras at roots of unity and crystal bases of quantum affine
  algebras}.
 {\em Comm. Math. Phys.}, 181:205--263, 1996.


\bibitem{Lo}
I.Losev, 
 Supports of simple modules in cyclotomic Cherednik categories O. Adv. Math. 377 (2021), article 107491.



\bibitem{MW}
C. M\oe glin and J.-L. Waldspurger, Sur l'involution de Zelevinski, J. Reine Angew. Math. 372 (1986), 136--177.


\bibitem{Mu}
G. Mullineux,
Bijections of $p$-regular partitions and $p$-modular irreducibles of the symmetric groups.
J. London Math. Soc. (2) 20 (1979), no.1, 60--66.


\bibitem {Xu1} M. Xu, On $p$-series and the Mullineux conjecture.
Comm. Algebra 27 (1999), no. 11, 5255--5265.

\bibitem{Xu2}  M. Xu,  On Mullineux' conjecture in the representation theory of symmetric groups. Comm. Algebra 25 (1997), no. 6, 1797--1803.


\end{thebibliography}
\end{document}